\documentclass[12pt]{article}

\usepackage{amsmath, amsfonts,amssymb,  amscd, amsthm, verbatim, mathtools, hyperref} 
\usepackage{fullpage, verbatim}
\usepackage[all]{xy}

\DeclareMathOperator{\ch}{char}

\def\A{\mathbb{A}}

\def\F{\mathbb{F}}

\theoremstyle{plain}

\newtheorem{thm}{Theorem}
\newtheorem{Lem}[thm]{Lemma}
\newtheorem{Cor}[thm]{Corollary}
\newtheorem{proposition}[thm]{Proposition}
\newtheorem{Conj}[thm]{Conjecture}

\theoremstyle{remark}
\newtheorem{Rem}[thm]{Remark}

\newtheorem{Exa}[thm]{Example}

\renewcommand\footnotemark{}

\makeatletter
\newcommand{\subjclass}[2][2010]{%
  \let\@oldtitle\@title%
  \gdef\@title{\@oldtitle\footnotetext{#1 \emph{Mathematics subject classification.} #2}}%
}
\newcommand{\keyywords}[1]{%
  \let\@@oldtitle\@title%
  \gdef\@title{\@@oldtitle\footnotetext{\emph{Key words and phrases.} #1.}}%
}
\makeatother

\begin{document}

\title{Factorization type probabilities of polynomials with prescribed coefficients over a finite field}
\keyywords{Monodromy group, Chebotarev density theorem, generically \'etale morphism, factorization type, finite field}
\author{Kaloyan Slavov
\thanks{This research was
supported by NCCR SwissMAP of the SNSF.}
}

\maketitle

\begin{abstract}
Let $f(T)$ be a monic polynomial of degree $d$ with coefficients in a finite field $\F_q$. 
Extending earlier results in the literature, but now allowing $(q,2d)>1$, we give a criterion for $f$ to satisfy the following property: for all but $d^2-d-1$ values of $s$ in $\F_q$, the probability 
that $f(T)+sT+b$ is irreducible over $\F_q$ (as $b\in\F_q$ is chosen uniformly at random) is $1/d+O(q^{-1/2})$. 
\end{abstract}



\section{Introduction}

Fix a positive integer $d$. Gauss proved that the probability for a monic polynomial of degree $d$ with coefficients in a finite field $\F_q$ to be irreducible is $1/d+O_d(q^{-1/2})$.
In fact, for any partition $\lambda=(\lambda_i)_{i=1}^k$ of $d$, the probability for a random monic $\F_q$-polynomial of degree $d$ to have exactly $k$ irreducible factors over $\F_q$ of degrees $\lambda_1,...,\lambda_k$ (i.e., to have factorization type $\lambda$) is $p_\lambda+O_d(q^{-1/2}),$ where $p_\lambda$ is the probability that a permutation in $S_d$ has cycle structure $\lambda$. 

The setting in which some coefficients of the polynomial are fixed and the remaining ones vary in $\F_q$ has been studied extensively. For a monic polynomial $f(T)\in\F_q[T]$ of degree $d$ and an integer $m$ with $0\leq m<d,$ it is conventional to define the $m$-th ``short interval'' in 
$\F_q[T]$ around $f$ to be
\[I(f,m)=\{f(T)+a_mT^m+\dots+a_1T+a_0\ |\ a_0,...,a_m\in\F_q\}.\]
We are particularly interested in the small cases $m=0,1$. One naturally asks for assumptions on $f$ under which the following expected statement holds true:

\bigskip

\hspace*{-0.7cm}
\begin{minipage}[t]{0.1\linewidth}
(*)
\end{minipage}%
\begin{minipage}[t]{0.9\linewidth}
For any partition $\lambda$ of $d$, the probability for an element of $I(f,m)$ to have factorization type $\lambda$ is $p_\lambda+O_d(q^{-1/2})$. 
\end{minipage}

\bigskip

While a ``sufficiently general'' polynomial $f\in\F_q[T]$ will satisfy (*) with $m=1$ or even with $m=0$, one is interested in an explicit criterion that can be used to check that a specific $f$ satisfies (*).  Along these lines, Bank, Bary-Soroker, and Rosenzweig  (\cite{Duke}) prove the following 

\begin{thm} Let $f(T)\in\F_q[T]$ be a monic polynomial of degree $d$. Suppose $(q,d(d-1))=1$. Then $f$ satisfies (*) with $m=1$. 
\label{DukeThm}
\end{thm}

A monic polynomial $f(T)\in\F_q[T]$ of degree $d$ is called a {\it Morse} polynomial if the equation $f'(T)=0$ has exactly $d-1$ distinct roots over $\overline{\F_q}$, and the values of $f$ at them are all distinct. 
For a Morse polynomial $f(T)$, (*) holds with $m=0$
(see \cite{book} or \cite{KR}). For $j\geq 0,$ the $j$-th Hasse derivative of a polynomial $f=\sum a_i T^i$ is defined as
\[D^jf=\sum \binom{i}{j}a_iT^{i-j},\]
so $f$ has a zero of order at least $k$ at $\alpha$ precisely when $D^jf(\alpha)=0$ for all $j=0,...,k-1$. 

The proposition below weakens the assumption $(q,d(d-1))=1$; it is stated as Proposition 7 in \cite{KR} and attributed to Jarden and Razon (Proposition 4.3 in \cite{book}).

\begin{proposition}
Let  $f(T)\in\F_q[T]$ be a monic polynomial of degree $d$. Suppose $f''\neq 0$ and $(q,2d)=1$. Then for all but $O_d(1)$ values of $s\in\F_q$, the polynomial $f(T)+sT$ is a Morse polynomial, and hence satisfies (*) with $m=0$. In particular, $f$ satisfies (*) with $m=1$.
\label{PropKR}
\end{proposition}

\begin{Rem} 
The assumption $(q,2d)=1$ is essential in Proposition \ref{PropKR}. Indeed, if $\ch \F_q\mid d$, a polynomial of degree $d$ is never a Morse polynomial. Also, even if the condition $f''\neq 0$ is replaced by the weaker $D^2f\neq 0$ (see the paragraph preceding Proposition 7 in  \cite{KR}), Proposition \ref{PropKR} still does not hold in characteristic $2$. For example, $T^7+sT$ is never a Morse polynomial when $q$ is a power of $2$, and in fact $f(T)=T^7$ fails to satisfy (*) with $m=1$. 
\end{Rem}

The goal of this note is to give a criterion for a polynomial to satisfy (*), but allowing 
$(q,2d)>1$. 

For a field $k$ and a polynomial $f(T)\in k[T]$, let $\widetilde{f}(x,y)$ denote the polynomial in 
$k[x,y]$ defined by
\[f(x)-f(y)=(x-y)\widetilde{f}(x,y).\]

We now state our main result. 

\begin{thm}
Let $f(T)\in\F_q[T]$ be a monic polynomial of degree $d$. Suppose 
$D^2f\neq 0$, $\deg f'\geq 1$, and
the polynomials $\widetilde{f}(x,y)-f'(x)$ and $\widetilde{f'}(x,y)$ have no common factors besides possibly a power of $x-y$.
  Then for all but 
$d^2-d-1$ values of $s\in\F_q$, the polynomial $f(T)+sT$ satisfies (*) with $m=0.$ 
\label{main}
\end{thm}

\begin{Cor} Let $f(T)\in\F_q[T]$ be a polynomial as in Theorem \ref{main}. Then $f$ satisfies (*) with $m=1$. 
\label{mainCor}
\end{Cor}

When $q$ is odd, Corollary \ref{mainCor} also follows from Corollary 1.4 in \cite{Entin}.

\begin{Exa}
Theorem \ref{main} and Corollary \ref{mainCor} apply to
$f(T)=T^{12}+T^3\in\F_q[T]$ with $q$ a power of $2$; the gcd of $\widetilde{f}(x,y)-f'(x)$ and 
$\widetilde{f'}(x,y)$ is $x-y$. 
\end{Exa}

\begin{Rem}
The statements of Theorem \ref{main} and Corollary \ref{mainCor} would be false if one drops the gcd assumption.  A counterexample is $f(T)=T^7$ in characteristic $2$. Thus Theorem \ref{main} here corrects the false Theorem 1.3 in our previous version \cite{ActaAR_wrong} of this paper. 
\end{Rem}

To apply Theorem \ref{main} to a specific polynomial, one has to compute the greatest common divisor of the two polynomials that appear in the statement; this task is computationally easy. In fact, based on modest numerical evidence, we state the following

\begin{Conj}
Let $k$ be a field and let $f\in k[T]$. Suppose $f''\neq 0$. Then the polynomials 
$\widetilde{f}(x,y)-f'(x)$ and $\widetilde{f'}(x,y)$ in $k[x,y]$ have no common factors. 
\end{Conj} 

In other words, we conjecture that the assumptions in Theorem \ref{main} not only cover further examples when $q$ is a power of $2$ or $\ch{\F_q}\mid d$ but are actually strictly weaker than the assumptions in Proposition \ref{PropKR}.

The proof of Theorem \ref{main} is based on the technique employed by Entin in a variety of problems solved in \cite{Entin}, with an extra ingredient  
(Lemma \ref{perturb} below) developed by the author in an earlier work, concerning the irreducibility of the perturbations of a certain curve. Namely, for $s\in\F_q,$ we set up a  generically \'etale map $\varphi_s\colon\A^1\to\A^1$ of degree $d$ such that for any $b\in\A^1(\F_q)$ with $d$ preimages over 
$\overline{\F_q}$, the conjugacy class in $S_d$ that the action of the Frobenius $\text{Fr}_q$
on $\varphi_s^{-1}(b)$ gives rise to has 
cycle structure corresponding to the factorization type
of the polynomial $f(T)+sT+b$ in $\F_q[T]$. The statement will then follow by the Chebotarev density theorem for function fields, once we show that the monodromy group of $\varphi_s$ is the full symmetric group $S_d$. 
To this end, we check the criterion proven in \cite{Ballico_Hefez}.

\section{The proof}

We say that a polynomial $f(T)\in\overline{\F_q}[T]$ is ``affine linearized'' if it has the form 
$f(T)=\sum a_i T^{p^i}+f(0)$, where $p=\text{char}\,\F_q$.

\begin{Lem}
Let $f(T)\in\F_q[T]$ be a polynomial of degree $d$, which is not affine linearized.
For all but at most $d-1$ values of 
$s\in\overline{\F_q}$, the polynomial 
$\widetilde{f}(x,y)+s$ is geometrically irreducible. 
\label{perturb}
\end{Lem}

\begin{proof}
The author has proven this as Lemma 19 in \cite{Slavov_Kakeya}. We sketch the proof here as well. First, an elementary undetermined coefficients argument shows that if $f$ is not affine linearized, the polynomial $\widetilde{f}(x,y)$ cannot be written as $Q(h(x,y))$ for a polynomial $Q$ with $\deg Q>1$. Then we apply the main result of \cite{DL}. 
\end{proof}

\begin{Lem}
Let $f\in\F_q[T]$ be a polynomial that satisfies the hypotheses of Theorem \ref{main}. 
  Then for all but 
$d^2-2d$ values of $s\in\overline{\F_q}$, there exists a $b\in\overline{\F_q}$ such that the polynomial $f(T)+sT+b$ has a unique root of multiplicity $2$ and $d-2$ simple roots over $\overline{\F_q}$. 
\label{adjust_poly_prescribed_fact_pattern}
\end{Lem}

\begin{proof} Let
\[B_1=\{-f'(\alpha)\ |\ D^2f(\alpha)=0\};\]
then $|B_1|\leq d-2$ and for any $s\notin B_1$ and $b\in\overline{\F_q}$, the polynomial $f(T)+sT+b$ has no roots over $\overline{\F_q}$ of multiplicity $3$ or more.    

Define 
\[X_1:=V(\widetilde{f}(x,y)-f'(x))\subset\A^2\quad \text{and}\quad  X_2:=V(\widetilde{f'}(x,y))\subset\A^2.\]
By B\'{e}zout's theorem, there are at most $(d-1)(d-2)$ pairs $(\alpha,\beta)\in (X_1\cap X_2)(\overline{\F_q})$ with $\alpha\neq\beta$.  Let
\[B_2=\{-f'(\alpha)\ |\ (\alpha,\beta)\in (X_1\cap X_2)(\overline{\F_q})\ \text{for some}\ \beta\in\overline{\F_q},\beta\neq \alpha\};\] then
$|B_2|\leq (d-1)(d-2).$ Suppose $\alpha\neq\beta$ in $\overline{\F_q}$ are both roots of multiplicity at least $2$ of some polynomial $f(T)+sT+b$ with $s,b\in\overline{\F_q}$. Explicitly, 
$f(\alpha)+s\alpha+b=f(\beta)+s\beta+b=0$ and $f'(\alpha)+s=f'(\beta)+s=0$. These imply 
$(\alpha,\beta)\in (X_1\cap X_2)(\overline{\F_q})$, hence $s\in B_2$. 

The set $B:=B_1\cup B_2$ satisfies $|B|\leq d^2-2d$. Let $s\notin B$. Choose $\alpha\in\overline{\F_q}$ such that $f'(\alpha)+s=0$ and set $b:=-f(\alpha)-s\alpha$. The polynomial $f(T)+sT+b$ satisfies the requirement. 
\end{proof}

\begin{proof}[Proof of Theorem \ref{main}.]
For any $s\in\F_q$, define
\[\Omega_s:=\{(t,b)\in\A^2\ |\ f(t)+st+b=0 \}.\]
The projection $\Omega_s\to\A^1_t$, $(t,b)\mapsto t$ is an isomorphism and the map 
$\varphi_s\colon\Omega_s\to\A^1_b$, $(t,b)\mapsto b$ is a generically \'etale morphism of degree $d$ between geometrically irreducible $\F_q$-varieties. 

The polynomial
$f(T)$ is not affine linearized, since
$\deg f'\geq 1$. Combining Lemma \ref{perturb} and Lemma \ref{adjust_poly_prescribed_fact_pattern}, there exists a set $B\subset\F_q$ of cardinality at most $d^2-d-1$ such that for any $s\in\F_q-B$, the following hold:
\begin{itemize}
\item[(i)] the polynomial $\widetilde{f}(x,y)+s$ is geometrically irreducible, and
\item[(ii)] there exists a $b\in\overline{\F_q}$ such that the polynomial $f(T)+sT+b$ has a unique root
of multiplicity $2$ and $d-2$ simple roots over $\overline{\F_q}$. 
\end{itemize} 

Let $s\in\F_q-B$.
By (ii), the fiber of
$\varphi_s$ over some $b\in\overline{\F_q}$  consists of $d-1$ points
over $\overline{\F_q}$, with $\varphi_s$ being \'etale at $d-2$ of them. Thus the assumption of Proposition 3 in \cite{Ballico_Hefez} is satisfied. 
Moreover, the complement of the diagonal in 
$\Omega_s\times_{\A^1,\varphi_s}\Omega_s$ is isomorphic to 
\[\Delta^c:=\{(x,y)\in\A^2\ |\ \widetilde{f}(x,y)+s=0, x\neq y\}.\]
It is nonempty because we can pick a $\beta\in\overline{\F_q}$ such that
$f'(\beta)+s\neq 0$, set $b:=-f(\beta)-s\beta$ (so $\beta$ is a simple root of $f(T)+sT+b$), let $\gamma$ be any other root of $f(T)+sT+b$ (note: $\deg f\geq 2$, since $\deg f'\geq 1$) and observe that 
$(\beta,\gamma)\in \Delta^c$. Thus $\Delta^c$ is a nonempty open subset of
$V(\widetilde{f}(x,y)+s)$
and by (i) is geometrically irreducible. Therefore the assumption of Proposition 2 in \cite{Ballico_Hefez} is satisfied as well. We conclude that the geometric monodromy group of the map $\varphi_s$ is the full $S_d$. 

Let $U$ be a dense open subset of $\A^1_b$ such that $\varphi_s|_{\varphi_s^{-1}(U)}\colon\varphi_s^{-1}(U)\to U$ is finite and \'etale. The statement now follows from Theorem 3 in \cite{Entin}, which is a version of the Chebotarev density theorem for function fields.
\end{proof}

\begin{Rem}  We can also deduce Corollary \ref{mainCor} directly from the criterion in \cite{Ballico_Hefez}, without going through Theorem \ref{main}. Namely, consider
\[\Omega:=\{(t,s,b)\in\A^3\ |\ f(t)+st+b=0\}\]
and $\varphi\colon\Omega\to\A^2_{s,b}.$ If $\Delta$ and $\Delta'$ denote the diagonals of $\Omega\times_{\A^2_{s,b}}\Omega$ and $\A^1_t\times\A^1_t$ respectively, then the map
$\Omega\times_{\A^2_{s,b}}\Omega-\Delta\to\A^1_t\times\A^1_t-\Delta'$
is an isomorphism, so the source is geometrically irreducible. The existence of 
$(s,b)\in\A^2(\overline{\F_q})$ such that $f(T)+sT+b$ has a unique root of multiplicity $2$ and $d-2$ simple roots over 
$\overline{\F_q}$ follows from Lemma \ref{adjust_poly_prescribed_fact_pattern}. 
\end{Rem}

\subsection*{Acknowledgments} 
I thank Bjorn Poonen and Alexei Entin for comments.


\begin{thebibliography}{HD82}


\normalsize

\bibitem{Ballico_Hefez}
E.~Ballico, A.~Hefez, {\it On the Galois group associated to a generically \'etale morphism},  Communications in Algebra, {\bf 14} (5), 899-909, 1986.
\bibitem{Duke}
 E.~Bank, L.~Bary-Soroker, and L.~Rosenzweig, {\it Prime polynomials in short
intervals and in arithmetic progressions}, 
Duke Math. J., {\bf 164} (2):277-295, 2015.
\bibitem{Entin} A.~Entin, {\it Monodromy of hyperplane sections of curves and decomposition statistics over finite fields}, International Mathematics Research Notices, rnz120, https://doi.org/10.1093/imrn/rnz120,
 arXiv:1805.05454v2.
\bibitem{book}
M. Jarden and A. Razon. Skolem density problems over large Galois extensions of global fields. In Hilbert’s
Tenth Problem: Relations with Arithmetic and Algebraic Geometry: Workshop
on Hilbert’s Tenth Problem: Relations with Arithmetic and Algebraic Geometry, November 2-5, 1999, Ghent University, Belgium, volume 270, page 213.
American Mathematical Soc., 2000. 
\bibitem{KR} P.~Kurlberg, L.~Rosenzweig, 
{\it Prime and M\"obius correlations for very short intervals in $\F_q[x]$},  arXiv:1802.01215.
\bibitem{DL} D. Lorenzini, {\it Reducibility of polynomials in two variables},  J. of Algebra, {\bf 156} (1993), 65-75.
\bibitem{Slavov_Kakeya} K.~Slavov,
{\it An algebraic geometry version of the Kakeya problem},
Finite Fields and Their Applications, 
{\bf 37} (2016), 158-178.
\bibitem{ActaAR_wrong} K.~Slavov, 
{\it Factorization type probabilities of polynomials with prescribed coefficients over a finite field}, Acta Arithmetica  {\bf 194} (2020), 315-318.
\end{thebibliography}
\end{document}